\documentclass[12pt,a4paper]{amsart}
\usepackage{amsmath}
\usepackage{amsthm}
\usepackage{amssymb}
\usepackage{mathrsfs}
\usepackage[english]{babel}
\usepackage{ot1patch}

\newtheorem{thm}{Theorem}[section]

\newtheorem{cor}[thm]{Corollary}
\newtheorem{prop}[thm]{Proposition}

\theoremstyle{remark}
\newtheorem{rmk}[thm]{Remark}
\newtheorem*{rmk*}{Remark}

\theoremstyle{definition}
\newtheorem{dfn}[thm]{Definition}

\numberwithin{equation}{section}

\title{Medial Axis detects non-Lipschitz Normally Embedded points}
\author{Adam Białożyt}
\date{May 2024}

\begin{document}
\begin{abstract}
This article demonstrates that every point where a closed set $X\subset \mathbb{R}^n$ is not Lipschitz Normally Embedded is approached by the medial axis of $X$.
\end{abstract}

\maketitle

\section{Introduction}

Inspired by the methods of Kosiba \cite{Kosiba} and Wolter \cite{Wolter}:

\subsection{Lipschitz Normal Embedding}
Consider a path-connected closed set $X\in \mathbb{R}^n$. We can introduce a metric on $X$ in several natural ways. The first approach is to induce the Euclidean metric from $\mathbb{R}^n$. Another one is to define the distance between $x,y \in X$ as the infimum of lengths of all rectifiable paths in $X$ joining $x,y$. We call set $X$ \textit{Lipschitz normally embedded - LNE} (as introduced by L. Birbrair and T. Mostowski \cite{BirbrairMostowski}) if these two metrics are equivalent. We will denote the inner metric as $d_{inn}(x,y)$. For the induced one, we will stick to the $\|x-y\|$ symbol.
We say that a set $X$ is \textit{Lipschitz Normally Embedded at a point }$p\in X$ if there exists $U$ - an open neighbourhood of $p$ such that $U\cap X$ is Lipschitz Normally Embedded. 

The study of Lipschitz Normally Embedded sets is usually undertaken in the context of semi-algebraic or definable sets. However, none of our results demands $X$ to be 'tamed'. Therefore, we present the results without this customary assumption.

\subsection{Medial Axis}
For a closed, nonempty subset $X$ of $\mathbb{R}^n$ endowed with the Euclidean norm, we define the distance of a point $a\in\mathbb{R}^n$ to $X$ by 
$$d(a,X)=d_X(a):= \inf\lbrace \|a-x\| \colon\,x\in X\rbrace,$$
which allows us to introduce the set of closest points in $X$ to $a$ as
$$m_X(a):=\lbrace x\in X\mid d(a,X)=\|a-x\|\rbrace.$$
We will drop the indices of the (multi-)functions $d$ and $m$.

Points where $m(a)$ consists of more than one point are conveniently collected in the \textit{medial axis} of $X$ denoted by $M_X$. That is:
$$M_X:=\lbrace a\in\mathbb{R}^n|\, \# m(a)>1\rbrace.$$
Following from the definition the multifunction $m$ is univalent in $\mathbb{R}^n\backslash M_X$. Moreover, it turns out that the medial axis is precisely the set of discontinuity of $m$ (see \cite{BirbrairDenkowski}). 

\subsection{Premise}
The distance function plays a crucial role in the definitions of LNE and MA. It seems attractive to check if this common characteristic bears any consequences. Non-Lipschitz Normally Embedded points of a set appear somewhat pinched in the ambient space. Such a shape is apocryphally believed to attract the medial axis. Like all statements in mathematics, this one should deserve its rigorous proof as well.

\section{The result}
Let us start by formally stating the definitions used later.
\begin{dfn}[Inner metric]
    Let $X$ be a path-connected closed subset of $\mathbb{R}^n$. For any $x,y$ set $\Gamma_{x,y}$ to be a set of paths joining $x$ and $y$, and let $len(\gamma)$ denote the length of a curve $\gamma$. If every $\Gamma_{x,y}$ contains at least one rectifiable curve we call a function \[d_{inn}: X\times X\ni (x,y) \to d_{inn}(x,y)=\inf \{len(\gamma)\mid \gamma \in \Gamma_{x,y}\}\]
    the \textit{inner metric on $X$}.
\end{dfn}

\begin{dfn}[Lipschitz Normal Embeddings]
    Let $X$ be a set with a properly defined inner metric. We call $X$ to be \textit{Lipschitz Normally Embedded} (or \textit{LNE}) if the inner metric on $X$ and the induced metric are equivalent. Since the inner metric is always bigger or equal to the induced one, $X$ is LNE exactly when there exists a constant $C>0$ such that \[d_{inn}(x,y)<C\|x-y\|.\]
    We call $X$ to be \textit{Lipschitz Normally Embedded at $p\in X$} if one can find $U$ - an open neighbourhood of $p$ such that $X\cap U$ is Lipschitz Normally Embedded.
\end{dfn}

\begin{dfn}[Medial Axis]
    
Let $X$ be a closed, nonempty subset of $\mathbb{R}^n$ endowed with the Euclidean norm, we define the \textit{distance of a point $a\in\mathbb{R}^n$ to $X$} by 
$$d(a,X)=d_X(a):= \inf\lbrace \|a-x\| \colon\,x\in X\rbrace,$$
which allows us to introduce the \textit{ set of closest points in $X$ to $a$} as
$$m_X(a):=\lbrace x\in X\mid d(a,X)=\|a-x\|\rbrace.$$
We will drop the indices of the (multi-)functions $d$ and $m$.

The \textit{medial axis} of $X$ denoted by $M_X$ is the set of points of $\mathbb{R}^n$ admitting more than one closest point in the set $X$, namely
$$M_X:=\lbrace a\in\mathbb{R}^n|\, \# m(a)>1\rbrace.$$
\end{dfn}

The main result of our discussion follows from the Lipschitz nature of the closest points function. The approach presented stems from F-H. Wolter. In \cite{Wolter} Theorem 4.2 he presented a characterisation of medial axis points through the convergence of the distance function gradient. That approach however falls short for our needs. We need the following strengthening of Wolter's result. 

\begin{prop}
 Let $X\subset \mathbb{R}^n$ be closed. Then the function $m(x)$ restricted to $\mathbb{R}^n\backslash \overline{M_X}$ is locally Lipschitz.
\end{prop}
\begin{proof}
    Take any open, convex and relatively compact $V$ in $\mathbb{R}^n\backslash \overline{M_X}$ such that $\overline{V}\cap \overline{M_X}=\varnothing$. Set $\delta = \frac{1}{2}\inf_{x\in V}d(x,M_X)$. To prove the proposition it suffices to show the existence of a constant $C>0$ such that for any pair of points $p$ and $q$ satisfying $\|p-q\| < \delta$ there is $\|m(p)-m(q)\|<C\|p-q\|$. 
    Indeed, having two points $p,q\in V$ we can divide the segment $[p,q]$ into subsegments $[p_i,p_{i+1}],\,i=0,\ldots,r$ with $p_0=p,\,p_r=q$ and $\|p_i-p_{i+1}\|<\delta$. Then with the constant $C$ as above we can estimate $\|m(p)-m(q)\|\leq \sum \|m(p_i)-m(p_{i+1})\|\leq \sum C\|p_i-p_{i+1}\|= C\|p-q\|$.

    The construction of the constant $C$ is slightly different for points in $X$ and those outside of it. 
    Let us start with the more complicated case and pick $p\in V\backslash X$.
    
    Firstly, let us consider points $q$ such that $\langle p-q,p-m(p)\rangle = 0$. Set $r= \|p-m(p)\|,\;\varepsilon = \|p-q\|$ and observe that $m(q)$ must be a point in $D:= B(q,\sqrt{r^2+\varepsilon^2})\backslash B(p_\delta,\delta+r)$, where $p_\delta$ is a point $p$ pushed away from $X$ by a distance $\delta$ (mind that $m(p_\delta) = m(p)$.) We can find an upper bound for $\|m(p)-m(q)\|$ by considering $\max_{x\in D} d(m(p),x)$. 
    
    It is easy to see that the maximum realising point must lie on the affine plane $A$ spanned by points $m(p),p$ and $q$. It is, in fact, one of the points in the intersection $A\cap  S(q,\sqrt{r^2+\varepsilon^2})\cap S(p_\delta,\delta+r)$ the other being $m(p).$ Denoting this point by $\xi$, we can see that triangles $p_\delta,m(p),\xi$ and $q,m(p),\xi$ are isosceles, and a line passing through points $p_\delta$ and $q$ is an altitude of these triangles. Denote by $s$ the foot of that altitude. 

    The triangles $p_\delta, p,q$ and $p_\delta, s,m(x)$ are similar so 
    \[\|m(p)-s\| = \varepsilon\frac{\delta+r}{\sqrt{\delta^2+\varepsilon^2}}\leq \varepsilon\frac{\delta+r}{\delta}.\]
    Now, setting $C = 2\frac{\delta+\sup_Vd(x,X)}{\delta}$ we can see that
    \[\|m(p)-m(q)\|\leq 2\|m(p)-s\| \leq C\|p-q\|.\]

    Secondly, let us turn to points $q$ such that $\langle p-q,p-m(p)\rangle \neq 0$. We can repeat the reasoning above substituting $p$ with $p'$ -- the projection of $q$ on the affine line $m(p)+\mathbb{R}(p-m(p))$. Indeed in such a case $d(p ',M_X)>\delta$ so $p '$ can be pushed to $p '_\delta$ and $\|p'-q\|\leq\|p-q\|<\delta$. Moreover, since $m(p)=m(p')$, we obtain 
    \[\|m(p)-m(q)\|=\|m(p')-m(q)\|\leq C\|p'-q\|\leq C\|p-q\|.\]

    Lastly, it remains to prove the inequality $\|m(p)-m(q)\|<C\|q-p\|$ for points $p\in X\cap V$ (possibly with some other universal constant). However, it turns out that for $p \in X$ we always have $d(m(p),m(q))\leq 2\|p-q\|$; no need for the additional assumption $q\notin M_X$. Truly, for any $\xi\in m(q)$, the triangle inequality gives \[ \|m(p)- \xi \|\leq \|m(p)-q\|+\|q - \xi\|.\]
    
    Now, since $p\in X$, there is $m(p) = p$ and $\|q-\xi\|\leq \|q-p\|$. 
    That shows \[ \max_{\xi\in m(q)}\|m(p)-\xi\|\leq 2\|p-q\|\] for $p\in X\cap V$ and proves the proposition.
\end{proof}

\begin{rmk}
    We cannot obtain the Lipschitz condition for $m$ in any neighbourhood of a point $x\in M_X$, as the multifunction $m$ ceases to be continuous on $M_X$. Truly, for any such point, there exists at least two distinct $x_1,x_2\in m(x)$ and by following the segments $[x_1,x]$ and $[x_2,x]$ we can check that $\|m(u)-m(v)\|$ fails to converge to zero for $(u,v)\to x$.
\end{rmk}

\begin{rmk}
    We can use the proposition to find points in $\overline{M_X}\backslash M_X$. Consider \[X=\{(\cos(u),\sin(u),u^2)\mid u\in  \mathbb{R}\}.\] Then, for every point $x_t=(0,0,t),\,t\geq 0$, its closest point in $X$ is equal to $m(x_t) = (\cos\sqrt{t},\,\sin\sqrt{t},\,t)$.
    The gradient $\nabla m|_L(x_t)$ (where $L$ denotes the line $\mathbb{R}x_1$) is equal to \[\nabla m|_L(x_t) = (-\frac{1}{2\sqrt{t}}\sin \sqrt{t},\, \frac{1}{2\sqrt{t}}\cos \sqrt{t},\, 1).\] It is unbounded in the neighbourhood of $t=0$, so $m$ fails to be locally Lipschitz there and we conclude that the point $(0,0,0)\in \overline{M_X}.$

    The example can be turned definable when one considers \[X' = \{x^2+y^2=1, z=x^2,y\geq 0,x\geq 0\}.\] However, the calculations for $X'$ become slightly less transparent.
    
\end{rmk}

Since the gradient of the distance function is expressed by a formula $\nabla d (x) = \frac{x-m(x)}{d(x)}$ we obtain the following

\begin{cor}
The gradient of the distance function $d_X$ is locally Lipschitz in $\mathbb{R}^n\backslash(\overline{M_X}\cup X)$.     
\end{cor}
The main consequence of the proposition is

\begin{thm}
    Let $X\subset \mathbb{R}^n$ be a closed set with a properly defined inner metric. Then for any $p\in X$ there is $p\in \overline{M_X}$ or $X$ is Lipschitz Normally Embedded in $p$. 
\end{thm}
\begin{proof}
    Assume $p\notin \overline{M_X}$, then there exists $r>0$ such that $\overline{B(p,r)}\cap \overline{M_X} = \varnothing$. 
    We will prove that there exists $C>0$ such that for all $x,y \in B(p,r)$ we have $d_{inn}(x,y)<C\|x-y\|$.
    
    For any choice of $x,y\in B(p,r)$ set $I$ to be a segment $[x,y]$. We have $I\cap M_X = \varnothing$, so $m(I)$ is a curve in $X$. Moreover $d_{inn}(x,y)\leq len(m(I))$ where $len$ denotes the length of a curve.

    Now, thanks to Proposition we know that $m(x)$ is Lipschitz on $\overline{B(p,r)}$ and its restriction to $I$ is almost everywhere differentiable. Denote by $C$ a constant such that $\|m(x)-m(y)\|<C\|x-y\|$ in $B(p,r)$. Setting $\mathcal{D_I}(m)$ to be the set of differentiability of $m|_I$, we can estimate the length of $m(I)$ as follows
    \[len(m(I)) = \int_{I\cap \mathcal{D_I}(m)} \|\nabla m|_I(t)\| dt\leq C\cdot len(I)  = C\|x-y\|.\]
    Hence 
    \[d_{inn}(x,y)<C\|x-y\|.\]
\end{proof}
\begin{rmk}
    Let $p\in X$ and $B(p,r)$ be separated from $M_X$. From the proof of the previous theorem we can see that points $x,y\in B(p,r)\cap X$ can be linked with at least one rectifiable curve from $\Gamma_{x,y}$. Therefore, the inner metric on $X$ can be defined locally around $p$.
\end{rmk}

\begin{rmk}
    However tempting the full equivalence ($p\notin \overline{M_X}$) $\iff$ ($X$ is LNE in $p$) is, we cannot expect it. The horn-like set $X = \{(x,y,z)| z=\sqrt{x^2+y^2}\}$ shatters our hopes. It is Lipschitz Normally Embedded yet $0 \in \overline{M_X}\cap X.$   
\end{rmk}

\section{Further Study}
It remains to check if we can always generate the sequence of points in $X$ that proves its nLNE by taking $x_n,y_n\in m(\xi_n)$ where $\xi_n\in M_X$. It seems plausible...

\bibliographystyle{abbrv}
\bibliography{references1}

\begin{thebibliography}{1}

\bibitem{BirbrairDenkowski}
L.~Birbrair and M.~Denkowski.
\newblock Medial axis and singularieties.
\newblock {\em J. Geom. Anal.}, 27(3):2339--2380, 2017.

\bibitem{BirbrairMostowski}
L.~Birbrair and T.~Mostowski.
\newblock Normal embeddings of semialgebraic sets.
\newblock {\em Michigan Mathematical J.}, 47(5):125--132, 2000.

\bibitem{Kosiba}
M.~Kosiba.
\newblock Lipschitz normally embedded sets do not need to have lipschitz normally embedded medial axis, 2024.

\bibitem{Wolter}
F.-E. Wolter.
\newblock {\em Cut loci in bordered and unbordered Riemannian manifolds}.
\newblock PhD thesis, Fechbereich Mathematik der Technischen Universitat Berlin, 1985.

\end{thebibliography}

\end{document}